\documentclass[paper=a4]{scrartcl}
\usepackage{amsmath,amssymb,amsthm,breqn,graphicx}
\usepackage[numbers,compress]{natbib}
\usepackage[T1]{fontenc}
\usepackage[utf8]{inputenc}
\usepackage{lmodern}
\usepackage{microtype}
\usepackage{algorithm}
\usepackage{algorithmic}
\usepackage{supertabular}

\theoremstyle{plain}
\newtheorem{lem}{Lemma}
\newtheorem{prop}[lem]{Proposition}
\newtheorem{thm}[lem]{Theorem}
\newtheorem{cor}[lem]{Corollary}
\theoremstyle{definition}

\newtheorem{example}{Example}
\theoremstyle{remark}
\newtheorem{rem}[lem]{Remark}
\newcommand{\D}{\mathbb D}
\newcommand{\PS}{\mathbb P}
\newcommand{\R}{\mathbb R}
\newcommand{\HS}{\mathbb H}
\newcommand{\qi}{\mathbf{i}}
\newcommand{\qj}{\mathbf{j}}
\newcommand{\qk}{\mathbf{k}}
\newcommand{\qe}{\overline}
\newcommand{\eps}{\epsilon}

\DeclareMathOperator{\primal}{primal}

\usepackage{hyperref}
\hypersetup{
  pdfauthor={Gabor Hegedus, Josef Schicho and Hans-Peter Schröcker},%
  pdftitle={Factorization of Rational Curves in the Study Quadric and Revolute Linkages},%
  pdfborder={0 0 0},%
  colorlinks=false,%
}

\title{Factorization of Rational Curves in the Study Quadric and Revolute Linkages}

\author{%
  Gabor Hegedüs\thanks{Johann Radon Institute for
    Computational and Applied Mathematics Austrian Academy of Sciences
    (RICAM), Altenbergerstrasse 69, 4040 Linz, Austria, e-mail:
    \{gabor.hegedues\}\{josef.schicho\}@oeaw.ac.at RICAM Linz},
  \addtocounter{footnote}{-1} %
  Josef Schicho\footnotemark{}, and %
  Hans-Peter Schröcker\thanks{University Innsbruck, Unit Geometry
    and CAD, Technikerstra\ss e 13, 6020 Innsbruck, Austria, e-mail:
    hans-peter.schroecker@uibk.ac.at}}

\begin{document}

\maketitle

\begin{abstract}
  Given a generic rational curve $C$ in the group of Euclidean
  displacements we construct a linkage such that the constrained
  motion of one of the links is exactly $C$. Our construction is based
  on the factorization of polynomials over dual quaternions. Low
  degree examples include the Bennett mechanisms and contain new types
  of overconstrained 6R-chains as sub-mechanisms.
\end{abstract}



\section{Introduction}

The research on this paper started with an attempt to understand the
geometry of Bennett linkages
\cite{bennett03,bennett14,bottema90,BHS,krames37,perez02} from the
point of view of dual quaternions. The group of Euclidean
displacements can be embedded as an open subset of the Study quadric
in the projectivization of the dual quaternions regarded as a real
vector space of dimension eight. Rotation subgroups and and their
composition then get an algebraic meaning (see Hao \cite{Hao}, Selig
\cite[Chapter~9 and 10]{selig05}). This was exploited in \cite{BHS} to
devise an algorithm for the synthesis of a Bennett linkage to three
pre-assigned poses. The key observation there was that the coupler
curve is the intersection of a unique 2-plane with the Study quadric.
Here, we translate the synthesis problem entirely into the language of
dual quaternions and we show that the problem is equivalent to the
factorization of a left quadratic polynomial into two linear ones.

Factorizations of left polynomials over the quaternions have been
studied by Niven in \cite{N}, who was interested in the number of such
factorizations, and Gordon and Motzkin in \cite{M}, who proved that
this number is either infinite or at most equal to the factorial of
the degree of the polynomial. (\cite{M} studies more generally
polynomials over central simple algebras over commutative fields.) The
more recent paper \cite{HS} by Huang and So gives an explicit solution
formula for quadratic polynomials. It is not difficult to extend these
results to dual quaternions.

Once the relation between the closure conditions in linkages and the
factorizations of left polynomials over dual quaternions became clear,
it also became clear that this relation holds for arbitrary linkages
with rotational joints. Given a generic rational curve $C$ of degree
$n$ on the Study quadric with parametrization $P(t)$, we construct
$n!$ different factorizations $P(t) = (t-h_1) \cdots (t-h_n)$ with
rotation quaternions $h_1,\ldots,h_n$ (or translation quaternions in
limiting cases). Each factorization corresponds to the movement of an
open $n$R-chain that guides the end-effector along $C$
(Section~\ref{sec:rational-curves}). Combining all open chains yields
an overconstrained mechanism whose combinatorial structure is
investigated in Section~\ref{sec:interchanging-factors}. In
Section~\ref{sec:prismatic-joints}, we discuss the extension to
linkages with prismatic joints.

The case $n = 2$ is just the Bennett case. Our results agree with the
recent findings of \cite{Ham} and naturally include the limit case of
Bennett linkages (RPRP linkages). For $n = 3$ we obtain linkages that
contain new examples of overconstrained 6R-chains. We present them in
more detail in Section~\ref{sec:overconstrained-6r}.

A preliminary short version of this article is \cite{hegedus12}. The
present article is more complete. We give strict proofs for all new
results, discuss the extension to prismatic joints, present additional
examples and illustrate our results by figures. The Maple source for
some of the examples and additional animations can be found on the
accompanying web-site
\url{http://geometrie.uibk.ac.at/schroecker/qf/}.

\section{Preliminaries}

In this section, we recall the well-known and classical description of
the group of Euclidean displacement by dual quaternions. The
presentation is already adapted to our needs in later sections. More
general references are \cite{Hao,selig05}.

We denote by $\mathrm{SE}_3$ the group of direct Euclidean
displacements, i.e., the group of maps from $\R^3$ to itself that
preserves distances and orientation. It is well-known that
$\mathrm{SE}_3$ is a semidirect product of the translation subgroup
$T$ and the special orthogonal group $\mathrm{SO}_3$, which may be
identified with the stabilizer of a single point.

We denote by $\D:=\R+\eps\R$ the ring of dual numbers, with
multiplication defined by $\eps^2=0$. The algebra $\HS$ is the
non-commutative algebra of quaternions, and $\D\HS$ is the algebra of
quaternions with coefficients in $\D$. Every dual quaternion has a
primal and a dual part (both quaternions in $\HS$), a scalar part in
$\D$ and a vectorial part in $\D^3$. The conjugated dual quaternion
$\qe{h}$ of $h$ is obtained by multiplying the vectorial part of $h$
by $-1$. The dual numbers $h\qe{h}$ and $h+\qe{h}$ are called
the \emph{norm} and \emph{trace} of $h$, respectively.

By projectivizing $\D\HS$ as a real 8-dimensional vectorspace, we
obtain $\PS^7$. The condition that the norm of $h$ is strictly real,
i.e.\ its dual part is zero, is a homogeneous quadratic equation. Its
zero set, denoted by $S$, is called the Study quadric. The linear
3-space of all dual quaternions with zero primal part is denoted by
$E$. It is contained in the Study quadric. The complement $S-E$ can be
identified with $\mathrm{SE}_3$. The primal part describes
$\mathrm{SO}_3$. Translations correspond to dual quaternions with
primal part $\pm 1$ and strictly vectorial dual part. More precisely,
the group isomorphism is given by sending $h=p+\eps q$ to the map
\begin{equation*}
  v \in \R^3 \mapsto\frac{pv\qe{p}+q\qe{p}}{p\qe{p}} \in \R^3.
\end{equation*}
The image of this map is strictly vectorial, the map is in
$\mathrm{SE}_3$, and the above formula indeed provides a group
homomorphism. Its bijectivity follows from the fact that both groups
are connected and of the same dimension.

A nonzero dual quaternion $h = p + \eps q$ represents a rotation if
and only if its norm and trace are strictly real ($h\qe{h}, h + \qe{h}
\in \R$) and its primal vectorial part is nonzero ($p \notin \R$). It
represents a translation if and only if its norm and trace are
strictly real and its primal vectorial part is zero ($p \in \R$). The
1-parameter rotation subgroups with fixed axis and the 1-parameter
translation subgroups with fixed direction can be characterized as
lines on $S$ through the identity element $1$. Translations are
characterized as those lines that meet the exceptional 3-plane~$E$.

\section{Rational motions and open linkages}
\label{sec:rational-curves}

In this section we study rational curves in the Study quadric and
construct a linkage with rotational joints such that the last link
moves along the prescribed curve. The main technique is the
factorization of polynomials over the dual quaternions.

\subsection{Polynomial factorization over dual quaternions}
\label{sec:factorization}

Let $C\subset S$ be a rational curve of degree $n>0$. Then there
exists a parametrization of $C$ by a polynomial
$(a_nt^n+a_{n-1}t^{n-1}+a_{n-2}t^{n-2}+\dots+a_0)_{t\in\PS^1}$, where
$a_0,\dots,a_n\in\D\HS$ and $\PS^1$ denotes the real projective line.
It is no loss of generality to assume that this polynomial is monic
($a_n = 1$). This can always be achieved by an appropriate choice of
coordinates.

Conversely, let $\D\HS[t]$ be the set of left polynomials with
coefficients over $\D\HS$. This set can be given a ring structure by
the convention that $t$ commutes with the coefficients. Let
$P\in\D\HS[t]$ be a polynomial of degree $n>0$. We call the map
$f_P\colon \PS^1\to\PS^7$, $t \mapsto P(t)$ the map associated to $P$.
The image is a rational curve of degree at most $n$.

If $Q\in\R[t]$, $Q\ne 0$, then the maps $f_P$ and $f_{PQ}$ are equal.
Conversely, if $P$ has a factor in $\R[t]$ of positive degree, we can
divide by it without changing the associated map. Note that, in
general, the set of right factors is different from the set of left
factors. However, a polynomial in $\R[t]$ or in $\D[t]$ is a left
factor if and only if it is a right factor, because it is in the
center of~$\D\HS[t]$.

For $n>0$, an open linkage with $n$ rotational joints can be described
algebraically as follows. Let $h_1,\dots,h_n$ be rotations; for each
$i$, the group parametrized by $(t-h_i)_{t\in\PS^1}$ -- the parameter
$t$ determines the rotation angle -- is the group of the $(i+1)$-th
link relative to the $i$-th link. If we choose the same parameter for
the $n$ rotations, then the position of the last link with respect to
the first link is given by a product
\begin{equation}
  \label{eq:factorization}
  P = (t-h_1)(t-h_2)\cdots(t-h_n)
  \quad\text{with}\quad
  t\in\PS^1.
\end{equation}
$P$ is a monic polynomial of degree $n$. Because it describes a rigid
motion, it satisfies the Study condition $P\qe{P} \in \R[t]$. In this
case we call $P$ a \emph{motion polynomial.} We ask the converse
question: Given a monic motion polynomial $P$ of degree $n$, is it
possible to construct a factorization of type
\eqref{eq:factorization}? We will see that, in general, the answer is
positive (Theorem~\ref{thm:open}, below).

For the time being we restrict ourselves to factorizations with
rotation quaternions only. Later, in
Section~\ref{sec:prismatic-joints}, we will show how to incorporate
translation quaternions as well.

\begin{lem}
  \label{lem:real}
  If $h$ is a rotation quaternion then $M := (t-h)(t-\qe{h})$ is in
  $\R[t]$ and has no real roots.
\end{lem}

\begin{proof}
  The first claim follows from the expansion $M = t^2 - t(h+\qe{h}) +
  h\qe{h}$ and the observations that $h + \qe{h} \in \R$ (because
  the scalar dual part of $h$ is zero) and $h\qe{h} \in \R$ (because
  $h$ lies on the Study quadric).

  The discriminant of $M$ equals $\Delta = (h+\qe{h})^2 - 4h\qe{h} =
  (h-\qe{h})^2$. Since $h$ is a rotation quaternion, $\Delta$ is
  negative and $M$ has no real roots. (The case $\Delta = 0$
  characterizes translations.)
\end{proof}

\begin{prop}
  \label{prop:open}
  Let $P\in\D\HS[t]$ be a motion polynomial of degree $n>0$ without
  strictly real factors. If there is a factorization
  $P(t)=(t-h_1)\cdots(t-h_n)$ with rotation quaternions
  $h_1,\ldots,h_n\in\D\HS$, the polynomial $P\qe{P}\in\R[t]$ has no
  real zeroes.
\end{prop}

\begin{proof}
  Assume $P=(t-h_1)\cdots(t-h_n)$ and let $M_i:=(t-h_i)(t-\qe{h_i})$
  for $i=1,\dots,n$. By Lemma~\ref{lem:real}, $M_i$ is in $\R[t]$ and
  has no real roots. The same is true for
  $P\qe{P}=(t-h_1)\cdots(t-h_n)(t-\qe{h_n})\cdots(t-\qe{h_1})=M_1\cdots
  M_n$.
\end{proof}

We call a motion polynomial $P\in\D\HS[t]$ of degree $n > 0$
\emph{generic} if $P\qe{P}$ has $n$ distinct quadratic, irreducible
factors. This implies that $\primal(P)$ has no strictly real factors,
because any such factor would appear with multiplicity $\ge 2$ in the
factorization of~$P\qe{P}$.

\begin{thm}
  \label{thm:open}
  Let $P\in\D\HS[t]$ be a generic motion polynomial of degree $n > 0$.
  Then there exists a factorization $P(t)=(t-h_1)\cdots(t-h_n)$ with
  $h_i\in\D\HS$ representing rotations.
\end{thm}

We divide the proof of Theorem~\ref{thm:open} into several lemmas. A
central tool (also for the computation of the factorization) is
polynomial division in~$\D\HS[t]$.

\begin{lem}[polynomial division]
  \label{lem:div}
  Let $P_1,P_2\in\D\HS[t]$ and assume $P_2$ is monic. Then there
  exists a unique representation $P_1=QP_2+R$ with $Q,R \in \D\HS[t]$
  and $\deg(R)<\deg(P_2)$. Moreover, if $h\in\D\HS$ such that
  $P_2(h)=0$, then $P_1(h)=R(h)$.
\end{lem}

\begin{proof}
  The proof is a generalization of polynomial division to the
  non-com\-mu\-ta\-tive case. Let $n:=\deg(P_1)$ and $m:=\deg(P_2)$.
  Existence of the representation is trivial if $n<m$ ($Q = 0$, $R =
  P_1$). Assume inductively that polynomial division is possible for
  all polynomials of degree less than $n$. We can write $P_1=at^n+P_3$
  with $\deg(P_3)<n$, and $at^n=at^{n-m}P_2+P_4$ with $\deg(P_4)<n$.
  By the induction hypothesis, we have $P_3=Q_3P_2+R_3$ and
  $P_4=Q_4P_2+R_4$ with $\deg(R_3),\deg(R_4)<m$. Combining, we obtain
  $P_1=(at^{n-m}+Q_4+Q_3)P_2+R_3+R_4$. This shows existence.

  Now, assume that $P_1=Q_1P_2+R_1=Q_2P_2+R_2$. Then
  $(Q_1-Q_2)P_2=R_1-R_2$. If $Q_1 \neq Q_2$, the polynomial on the
  left has degree greater or equal $m$ while the polynomial on the
  right has degree less than $m$. This is impossible, so that $Q_1 =
  Q_2$ and $R_1 = R_2$. This shows uniqueness.

  For the last statement, we have to show that $P_2(h)=0$ implies
  $(QP_2)(h)=0$. This is not trivial, because $h$ does not commute
  with the coefficients of $P_2$. But the statement is linear in $Q$,
  hence it is enough to prove it for monomials: If $P_2(h) = 0$ and
  $Q=at^r$, then $(QP_2)(h)=aP_2(h)h^r=0$.
\end{proof}

\begin{lem}
  \label{lem:factor}
  Let $P\in\D\HS[t]$ and $h\in\D\HS$. Then $(t-h)$ is a right factor
  of $P$ if and only if $P(h)=0$.
\end{lem}

\begin{proof}
  By Lemma~\ref{lem:div}, there is a unique representation
  $P=Q(t-h)+R$, and $R$ is a constant equal to $P(h)$.
\end{proof}

\begin{lem}
  \label{lem:root}
  Let $P\in\D\HS[t]$ be a motion polynomial. Let $M\in\R[t]$ be a
  monic polynomial of degree two that divides $P\qe{P}$ but not
  $\primal(P)$. Then there exists a unique $h\in\D\HS$ such that
  $P(h)=M(h)=0$.
\end{lem}

\begin{proof}
  By Lemma~\ref{lem:div}, we may write $P=QM+R$ with $Q,R\in\D\HS[t]$,
  $\deg(R)<2$. Since $M$ does not divide $\primal(P)$, we have
  $\primal(R) \neq 0$. On the other hand,
  \begin{equation*}
    P\qe{P}=(QM+R)(M\qe{Q}+\qe{R})=(Q\qe{Q}M+Q\qe{R}+R\qe{Q})M+R\qe{R}.
  \end{equation*}
  Because $M$ divides $P\qe{P}$, we conclude $R\qe{R}=cM$ for some
  $c\in\D$, $c\ne 0$. Assume that the primal part of $c$ is zero, that
  is, $R\qe{R}=\eps kM$ with $k\in \HS$. This implies $\primal(R) = 0$
  and contradicts our assumption. Hence $\primal(c) \neq 0$. The
  leading coefficient $r_1$ of $R$ cannot have zero primal part
  because otherwise we get the contradiction $r_1\qe{r}_1 = 0 = c$.
  Hence, $r_1$ is invertible in $\D\HS$ and, because of
  $r_1\qe{r}_1=c$, so is $c$. The polynomial $R = r_1t + r_0$ is
  linear with invertible leading coefficient. Hence it has a unique
  zero $h = -r_1^{-1}r_0 \in \D\HS$. From $R(h)=0$, we obtain
  $cM(h)=0$ and, because $c$ is invertible, $M(h)=P(h)=0$. This shows
  existence.

  In order to show uniqueness, assume there exists $h' \in \D\HS$ with
  $P(h')=M(h')=0$. This implies $R(h')=0$. But then $h' = h$ because
  the zero of $R$ is unique.
\end{proof}

\begin{proof}[Proof of Theorem~\ref{thm:open}:]
  We proceed by induction on $n$. For $n=0$, the statement is trivial.
  Assume $n\ge 1$. Since the primal part of $P$ has no strictly real
  factors, $P$ itself has no strictly real factors. Consequently,
  $P\qe{P}$ has no linear factors.

  Let $M$ be one of the irreducible quadratic factors of $P\qe{P}$. By
  Lemma~\ref{lem:root}, there is a unique $h$ such that $M(h)=P(h)=0$.
  By Lemma~\ref{lem:factor}, there exists $Q\in\D\HS[t]$ such that
  $P=Q(t-h)$. Obviously, $Q$ is monic of degree $n-1$. Moreover, we
  have $P\qe{P}=Q(t-h)(t-\qe{h})\qe{Q}=Q\qe{Q}M$ so that $Q\qe{Q}$ is
  in $\R[t]$. Furthermore, $Q$ cannot have a strictly real factor:
  This factor would also be a left factor and hence divide $P$. For
  similar reasons, the primal part of $Q$ cannot have strictly real
  factor (it would divide $\primal(P)$). By induction hypothesis, we
  obtain $Q=(t-h_1)\cdots(t-h_{n-1})$ and so
  $P=(t-h_1)\cdots(t-h_{n-1})(t-h)$. Because of $P\qe{P} \in \R[t]$,
  $h$ must be a rotation or translation quaternion. By genericity of
  $P$, it is actually a rotation quaternion.
\end{proof}

\begin{rem}
  Theorem~\ref{thm:open} is almost a converse to
  Proposition~\ref{prop:open}. But there are polynomials $P$ without
  real factors such that $P\qe{P}\in\R[t]$ where the proposition and
  the theorem do not say anything. In this case, there exists an
  irreducible polynomial $R\in\R[t]$ dividing the primal part of $P$
  but not the dual part. Since $R$ then also divides the primal part
  of $\qe{P}$ and $P\qe{P}\in\R[t]$, the factor $R$ appears with
  multiplicity two in the factorization of $P\qe{P}$. For instance, if
  $P=t^2+1+\eps \qi$, it can be shown that $P$ is not the product of
  two linear rotation polynomials. (Here, we assume that
  $(1,\qi,\qj,\qk)$ is the standard basis of the quaternion algebra
  $\HS$.) On the other hand, $P=t^2+1+\eps\qj t +
  \eps\qi=(t-\qk)(t-\qk+\eps\qj)$ is a product of two rotation
  polynomials. A systematic analysis would be good, but it is probably
  more difficult. At this place, we only observe that $P=t^2+1+\eps
  \qi$ is a quadratic parametrization of a straight line.
\end{rem}


\subsection{Computing the factorization}
\label{sec:computing}

Our proof of Theorem~\ref{thm:open} is constructive and, with
exception of the factorization of $P\qe{P}$, can be implemented in
rational arithmetic. We describe this in more detail. Our
considerations will lead us to two new insights:
\begin{itemize}
\item The factorization depends on an ordering of the $n$ quadratic
  factors of the polynomial $P\qe{P}$. Hence, there exist $n!$
  different factorizations and Theorem~\ref{thm:open} actually admits
  a stronger version (Theorem~\ref{thm:number}, below).
\item If one factorization of $P$ is known, the remaining
  factorizations can be computed in rational arithmetic. In
  particular, it is possible to construct completely rational examples
  by starting with the polynomial $P = (t-h_1)\cdots(t-h_n)$ where
  $h_1,\ldots,h_n$ are rotation quaternions with rational
  coefficients.
\end{itemize}

Given is a motion polynomial $P \in \D\HS[t]$. We want to compute
rotation quaternions $h_1,\ldots,h_n$ such that $P = (t-h_1) \cdots
(t-h_n)$. The pseudo-code for this calculation is given in
Algorithm~\ref{alg:factor}. It returns an $n$-tuple $H =
(h_1,\ldots,h_n)$. Here are some remarks on the actual calculations:

\begin{algorithm}[tbp]
  \caption{Compute factorization $P = (t-h_1)\ldots(t-h_n)$}
  \label{alg:factor}
  \begin{algorithmic}[1]
    \REQUIRE{generic motion polynomial $P$}
    \STATE{$M \leftarrow \{M_1,\ldots,M_n\}$ where $P\qe{P} = M_1 \cdots M_n$.}
    \STATE{Let $H$ denote the empty tuple.}
    \REPEAT
      \STATE{Choose $M_i \in M$ and set $M \leftarrow M \setminus \{M_i\}$.}
      \STATE{Compute $h_i$ such that $M_i(h_i) = P(h_i) = 0$.}
      \STATE{Append $h_i$ to $H$.}
      \STATE{$P \leftarrow P / (t-h_i)$}
    \UNTIL{$\deg P = 0$}
    \RETURN{$H = (h_1,\ldots,h_n)$}
  \end{algorithmic}
\end{algorithm}

\begin{itemize}
\item In Line~4, the choice of $M_i \in M$ is arbitrary. Thus,
  Algorithm~\ref{alg:factor} is not deterministic.
\item The common root $h_i$ of $M_i$ and $P$ in Line~5 is found by
  computing the remainder $R = r_1t + r_0$ of the polynomial division
  of $P$ by $M_i$ and setting $h_i = -r_1^{-1}r_0$. This is justified
  in the proof of Lemma~\ref{lem:root}.
\item The proof of Theorem~\ref{thm:open} shows that the polynomial
  division of $P$ by $(t-h_i)$ is possible without remainder.
  Moreover, the quotient satisfies all requirements of
  Theorem~\ref{thm:open} so that it is suitable as input for yet
  another repeat-loop.
\end{itemize}

Maple source code for Algorithm~1 can be found on the web-site
\url{http://geometrie.uibk.ac.at/schroecker/qf/}. The
non-deterministic nature of Algorithm~\ref{alg:factor} implies
existence of different factorizations into products of rotation
quaternions.

\begin{thm}
  \label{thm:number}
  Let $P\in\D\HS[t]$ be a generic motion polynomial of degree $n>0$.
  Then there is a one-to-one correspondence between factorizations
  $P(t)=(t-h_1)\cdots(t-h_n)$ into linear polynomials over $\D\HS$ and
  permutations of the $n$ distinct quadratic irreducible factors
  $M_1,\dots,M_n$ of $P\qe{P}$. In any of these factorizations we have
  $M_i(h_i)=0$ for $i=1,\dots,n$.
\end{thm}

\begin{proof}
  Our proof of Theorem~\ref{thm:open} can be translated into a
  construction of a factorization of $P$ into linear factors over
  $\D\HS$. The only non-deterministic step is the choice of a
  quadratic factor of $P'\qe{P'}$, where $P'$ is the left factor from
  which the next right linear factor is going to be constructed. The
  construction is also complete in the sense that every factorization
  can be obtained by this non-deterministic algorithm.
\end{proof}

\begin{example}
  \label{ex:1}
  We give an example that illustrates some of the results we obtained
  so far. Consider the quadratic polynomial
  \begin{equation*}
    P = t^2 -
    t(1 + (\eps-1)\qi + (1-\eps)\qj + 2(1+\eps)\qk) -
    1 - 2\eps + \qi - \eps\qj + (2 - \eps)\qk.
  \end{equation*}
  We can write $P\qe{P} = M_1M_2$ where
  \begin{equation*}
    M_1 = t^2+2
    \quad\text{and}\quad
    M_2 = t^2-2t+3.
  \end{equation*}
  In order to find the common root $h_2$ of $P$ and $M_1$, we compute
  the remainder of the polynomial division of $P$ by $M_1$. It is $R =
  r_1t + r_0$ where
  \begin{equation*}
    r_1 = -1 + (1-\eps)\qi + (\eps-1)\qj - 2(1+\eps)\qk,\quad
    r_0 = -3 - 2\eps + \qi - \eps\qj + (2-\eps)\qk.
  \end{equation*}
  Now
  \begin{equation*}
    h_2 = -r_1^{-1}r_0 = \Bigl(\frac{-4}{7}+\frac{30}{49}\eps\Bigr)\qi -
                        \Bigl(\frac{1}{7}+\frac{3}{49}\eps\Bigr)\qj +
                        \Bigl(\frac{9}{7}+\frac{13}{49}\eps\Bigr)\qk
  \end{equation*}
  and, by polynomial division, we find $P = (t-h_1)(t-h_2)$ with
  \begin{equation*}
    h_1 = 1 +
          \Bigl(\frac{-3}{7}+\frac{19}{49}\eps\Bigr)\qi +
          \Bigl(\frac{8}{7}+\frac{-46}{49}\eps\Bigr)\qj +
          \Bigl(\frac{5}{7}+\frac{85}{49}\eps\Bigr)\qk.
  \end{equation*}
  Note that $P(h_2) = M_1(h_2) = 0$.

  Using $M_2$ instead of $M_1$, we can compute a second factorization
  $P = (t-h'_1)(t-h'_2)$ where
  \begin{equation*}
    \begin{aligned}
      h'_1 &= (1-\eps)\qj + (1+\eps)\qk,\\
      h'_2 &= 1-(1-\eps)\qi + (1+\eps)\qk.
    \end{aligned}
  \end{equation*}
  Here, $P(h'_2) = M_2(h'_2) = 0$.
\end{example}

We already showed how to compute the rotation quaternion $h_i$ from
the factor $M_i$. But it is also possible to compute $M_i$ from $h_i$.
Clearly, $M_i$ is the unique monic polynomial of degree two such that
$M_i(h_i) = 0$ \emph{(the minimal polynomial of $h_i$).} It is given
by $M_i = t^2 - (h_i+\qe{h}_i)t + h_i\qe{h}_i$. This observation is
important, because it allows us to construct completely rational
examples. Setting $P = (t-h_1)\cdots(t-h_n)$ with
$h_1,\ldots,h_n\in\D\HS$ we can directly compute the minimal
polynomials $M_1,\ldots,M_n$ and also the factorization in Line~1 of
Algorithm~\ref{alg:factor}. The polynomials in our examples were
actually obtained in this way.

\subsection{Kinematic interpretation}
\label{sec:interpretation}

Now we are going to translate Theorem~\ref{thm:open} into the language
of kinematics. A precise formulation takes into account the
possibility of neighboring factors $t-h_i$ and $t-h_{i+1}$ that
describe rotations about the same axis. In this case, we call $h_i$
and $h_{i+1}$ \emph{compatible.} This is the case if and only if
$h_ih_{i+1} = h_{i+1}h_i$.

Let $h\in\D\HS$ be a dual quaternion representing a rotation. The
parametrization $(t-h)_{t\in\PS^1}$ of the rotation group defined by
$h$ is called a linear parametrization. More generally, let
$R_1,R_2\in\R[t]$ such that $R_1$ is monic, $\deg(R_1)=n$,
$\deg(R_2)<n$, without common factor. Then the parametrization
$(R_1(t)-h_1R_2(t))_{t\in\PS^1}$ is called a rational parametrization
of degree $n$. Higher degree parametrizations of rotation groups may
arise as the product of linear parametrizations, if the two axes
coincide, that is, if the rotation quaternions are compatible.

A rational curve $C\subset S$ of degree $n$ in the Study quadric
admits a parametrization by a motion polynomial $P$ of degree $n$. We
say that $C$ is a generic rational curve of degree $n$ in the Study
quadric if $P$ is generic (its norm polynomial has $n$ distinct
irreducible quadratic factors). It is straightforward to show that
this notion of genericity of $C$ is well-defined, i.e., it does not
depend on the choice of the motion polynomial~$P$.

\begin{cor} 
  \label{cor:open-chain}
  Let $C\subset S$ be a generic rational curve of degree $n$ in the
  Study quadric, passing through $1$. Then $C$ can be obtained as
  movement of the last link of an open $k$R-linkage, with $k\le n$.
  The rotations in the $k$ joints have a simultaneous rational
  parametrization, and the sum of the degree of these parametrizations
  equals~$n$.
\end{cor}

\begin{proof}
  Let $P\in\D\HS[t]$ be a motion polynomial of degree $n$ that
  parametrizes $C$. The primal part of $P$ cannot have real factors
  because $C$ is generic. By Theorem~\ref{thm:open}, there exists a
  factorization $P=(t-h_1)\cdots(t-h_n)$ with rotation quaternions
  $h_1,\dots,h_n$. Assume that $h_i,\dots,h_{i+m-1}$ are compatible.
  Because every dual quaternion compatible with $h_i$ is a real linear
  combination of $1$ and $h_i$, the product
  $(t-h_i)\cdots(t-h_{i+m-1})$ can be written as $R_1-h_iR_2$ for
  $R_1,R_2\in\R[t]$ such that $R_1$ is monic, $\deg(R_1)=m$,
  $\deg(R_2)<m$. This already implies the corollary's statement.
\end{proof}

\begin{rem}
  \label{rem:coinciding}
  It is sometimes advantageous to think of the open $k$R-chain
  referred to in Corollary~\ref{cor:open-chain} as an open $n$R-chain
  with the possibility of coinciding consecutive axes. See, for
  example, Theorem~\ref{thm:closed}, below.
\end{rem}

\begin{example}
  \label{ex:2}
  We give a kinematic interpretation of the calculation in Example~1.
  The two factorizations
  \begin{equation*}
    P = (t-h_1)(t-h_2) = (t-h'_1)(t-h'_2)
  \end{equation*}
  show that the motion parametrized by $P$ occurs as end-effector
  motion of two open 2R-chains, parametrized with the same rational
  parameter $t$. The axes of the rotation quaternions $h_2$, $h_1$,
  $k_1$, $k_2$ in the moving frame form a closed 4R-chain whose
  coupler motion is parametrized by $P$. Hence, we actually presented
  a method to synthesize a Bennett mechanism to three given poses.
  Without loss of generality, we assume that one of the poses is the
  identity. From this data, it is easy to compute the quadratic
  parametrization of the Bennett motion as in \cite{BHS}. It serves as
  input for our factorization algorithm. Note that in general a cubic
  algebraic number needs to be introduced for the factorization of the
  real quartic polynomial $P\qe{P}$. This has also been observed in
  other synthesis algorithms for Bennett linkages.
\end{example}

\section{Interchanging factors and closed linkages}
\label{sec:interchanging-factors}

In this chapter, we investigate the linkage formed by combining the
$n!$ open $n$R-chains that can be used to generate a generic rational
curve $C$ of degree $n$ on the Study quadric.

In order to describe the combinatorial structure of a linkage, we
recall the definition of the link graph. Recall that the rigid parts
of a linkage are called links, and the existence of a rotational joint
between two links means that the two links share a fixed line, the
rotation axis. It is possible that $m>2$ links share the same rotation
axis. In this case, we have $\binom{m}{2}$ different joints supported
at this axis. The link graph consists of a node for each link and an
edge for each joint connecting two links. This should not be confused
with the axis graph, which is also used in the literature. There, the
nodes correspond to the axes of joints and the edges correspond to
links.

If the linkage moves rationally with mobility one, for every link
there exists an algebraic function from $\PS^1$ to $\PS^7$, mapping a
time parameter $t\in\PS^1$ to the pose of the link at time $t$. The
pose can be represented by an element of $\mathrm{SE}_3$ or,
equivalently, by a point on the Study quadric. If two links $L_1$ and
$L_2$ are connected by a joint $J$, the relative motion of $L_2$ with
respect to $L_1$ is parametrized by the quotient
$\psi_J:=\phi_1\phi_2^{-1}$ of the two pose functions. This motion is
a rotation. Hence it also has parametrization by a motion polynomial
of degree one. The linkage we construct will have the property that
$\psi_J$ is a linear polynomial $(t-h_J)$.

Note that when $h_J$ is specified for all joints $J$, the relative
motions of any pair of links can be obtained by multiplication; the
linkage kinematics is fully determined. It is easy to extract the
Denavit-Hartenberg parameters from the rotation quaternions $h_J$. In
this paper we are content with the linkage specification by dual
quaternions.


\begin{thm}
  \label{thm:closed}
  Let $C\subset S$ be a generic rational curve of degree $n$ in the
  Study quadric, passing through $1$. Then $C$ is contained in the
  motion of a link in a mechanism with revolute joints, $2^n$ links
  and $n2^{n-1}$ joints. The link graph is the 1-skeleton of the
  $n$-dimensional hypercube.
\end{thm}

\begin{rem}
  In Theorem~\ref{thm:closed} we can only state that $C$ is contained
  in the motion of one link. It cannot be excluded that the mechanism
  has more than one degree of freedom. An extreme case arises from the
  product $P = (t-h_1)\cdots(t-h_n)$ of compatible rotation
  quaternions $h_1,\ldots,h_n$. With the understanding of
  Remark~\ref{rem:coinciding}, the corresponding link graph can still
  be considered as 1-skeleton of a hypercube but the linkage has $n$
  trivial degrees of freedom. A sufficient condition that ensures
  equality of $C$ and the link motion is that every $4$-cycle in the
  link graph corresponds to a non-degenerate Bennett linkage.
\end{rem}

\begin{proof}[Proof of Theorem~\ref{thm:closed}]
  Let $P(t)$ be a generic motion polynomial of degree $n$,
  parametrizing $C$. We denote the $n$ irreducible quadratic factors
  of $P\qe{P}$ by $M_1,\dots,M_n$ and construct a linkage whose nodes
  are labelled by the subsets of $G:=\{1,\dots,n\}$. Two nodes are
  connected by an edge if and only if the corresponding subsets differ
  by exactly one element (Figure~\ref{fig:linkage_graph}). This graph
  is known as Hamming graph.

  Let $F\subseteq G$ be a subset of cardinality $m \le n$. By
  successively dividing out right factors $(t-h_i)$ with $M_i =
  (t-h_i)(t-\qe{h}_i) \in F$, we obtain a factorization $P=UV$ with
  monic $U,V\in\D\HS[t]$, and $V\qe{V}$ is the product of the factors
  in $F$. This can be done in $m!$ ways. We claim that the result is
  always the same. To show this, we assume that we have two
  factorizations $P=U_1V_1=U_2V_2$ corresponding to two different
  permutations $\sigma_1,\sigma_2$ of $M_1,\dots,M_n$. In our
  situation, the first $m$ elements in $\sigma_1$ are a permutation of
  the first $m$ elements in $\sigma_2$, in fact this is the set $F$,
  and we assume that the remaining elements are equal. This is
  possible because we are free to choose any order in the remaining
  factors in $G \setminus F$. Now we apply our non-deterministic
  algorithm of dividing out right factors to $\qe{P}$, in the reverse
  order of the permutations. We obtain the two factorizations
  $\qe{P}=\qe{V_1}\qe{U_1}=\qe{V_2}\qe{U_2}$. But in these two
  division processes, the first $n-m$ choices are equal, and it
  follows that $\qe{U_1}=\qe{U_2}$. Consequently, we get $U_1=U_2$ and
  $V_1=V_2$. Hence $F$ determines the right factor $V$ (and also the
  left factor $U$) uniquely, and we place the link $L_F$ at position
  $V(t)$, depending on one real parameter $t$ but not on an ordering
  of $F$. This already proves the claim on the linkage graph.

  Now we show that linkage consists of revolute joints only. Assume
  that $F_1$ and $F_2$ differ by a single element $M_i$. Without loss
  of generality we may assume $F_2=F_1\cup\{M_i\}$. Let
  $P=U_1V_1=U_2V_2$ be the factorizations obtained as above. Then we
  have $V_2=(t-h_i)V_1$, where $h_i$ is unique common solution of
  $U_1$ and $M_i$ (see Lemma~\ref{lem:root}). Hence the relative
  position of $L_{F_2}$ with respect to $L_{F_1}$, depending on $t$,
  is a rotation group parametrized by $(t-h_i)_{t\in\PS^1}$. So we see
  that $L_{F_2}$ is connected with $L_{F_1}$ by a rotational joint.

  The two links whose relative movement is the curve $C$ are easy to
  determine: They are labelled by the empty set and by~$G$.
\end{proof}

\begin{figure}
  \centering
  \includegraphics{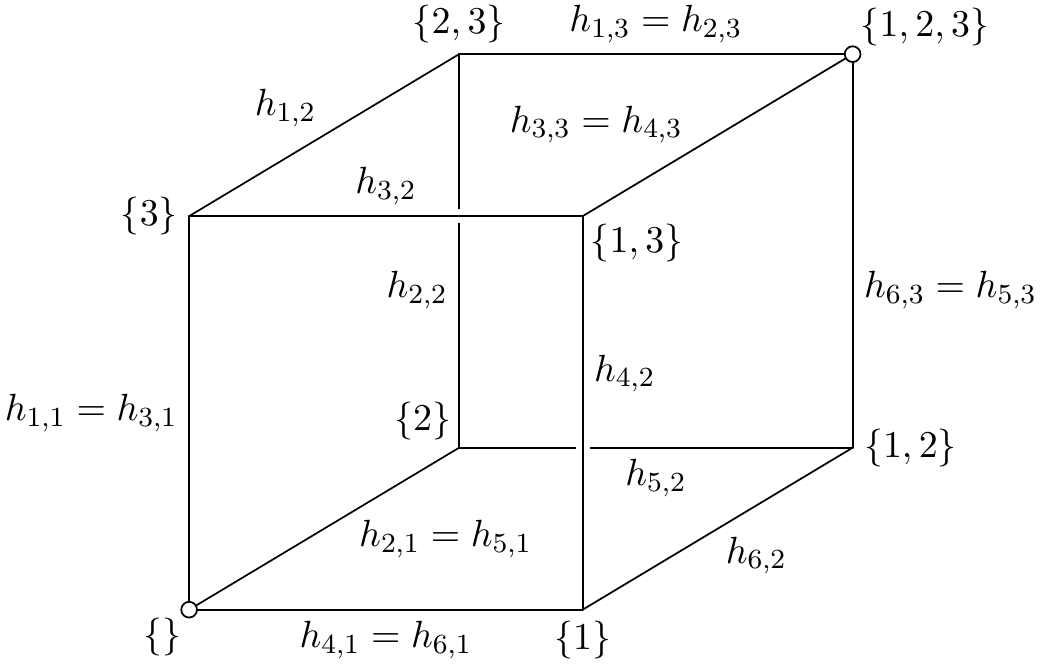}
  \caption{Linkage graph obtained by factorizing a cubic curve}
  \label{fig:linkage_graph}
\end{figure}

\begin{example}
  \label{ex:3}
  We present the construction of a linkage to the graph of
  Figure~\ref{fig:linkage_graph}. A parametrization of a rational
  cubic curve $C \subset S$ reads
  \begin{multline*}
    P = t^3 -
        t^2(3+(4-\eps)\qi+(1+3\eps)\qj+2(1+\eps)\qk) -\\
        t  (3(1+\eps)-3(3-\eps)\qi-(1+11\eps)\qj-(3+5\eps)\qk) +\\
          2(3 - (1-3\eps)\qi-(1+2\eps)\qj-(1+\eps)\qk).
  \end{multline*}
  The real polynomial $P\qe{P}$ factors as
  \begin{equation*}
    P\qe{P} = (\underbrace{t^2-2t+2}_{=:M_1})(\underbrace{t^2-2t+4}_{=:M_2})(\underbrace{t^2-2t+6}_{=:M_3}).
  \end{equation*}
  The polynomial $P$ admits six factorizations $P =
  (t-h_{l,1})(t-h_{l,2})(t-h_{l,3})$, one for each permutation of the
  triple $(M_1,M_2,M_3)$:
  \par\medskip
  \begin{supertabular}{ll}
      $(M_1,M_2,M_3)\colon$ & $\displaystyle h_{1,1} = 1+({\tfrac{65}{31}}-{\tfrac{814}{961}}\eps)\qi+({\tfrac{16}{31}}+{\tfrac{1373}{961}}\eps)\qj+({\tfrac{18}{31}}+{\tfrac{1719}{961}}\eps)\qk$                     \\[0.5ex]
                            & $\displaystyle h_{1,2} = 1+({\tfrac{395}{403}}-{\tfrac{94035}{162409}}\eps)\qi+({\tfrac{319}{403}}+{\tfrac{53380}{162409}}\eps)\qj+({\tfrac{41995}{162409}}\eps+{\tfrac{479}{403}})\qk$  \\[0.5ex]
                            & $\displaystyle h_{1,3} = 1+({\tfrac{12}{13}}+{\tfrac{72}{169}}\eps)\qi+({\tfrac{210}{169}}\eps-{\tfrac{4}{13}})\qj+({\tfrac{3}{13}}-{\tfrac{8}{169}}\eps)\qk$                            \\[1.0ex]
      $(M_1,M_2,M_3)\colon$ & $\displaystyle h_{1,1}  = 1+({\tfrac{65}{31}}-{\tfrac{814}{961}}\eps)\qi+({\tfrac{16}{31}}+{\tfrac{1373}{961}}\eps)\qj+({\tfrac{18}{31}}+{\tfrac{1719}{961}}\eps)\qk$                    \\[0.5ex]
                            & $\displaystyle h_{1,2}  = 1+({\tfrac{395}{403}}-{\tfrac{94035}{162409}}\eps)\qi+({\tfrac{319}{403}}+{\tfrac{53380}{162409}}\eps)\qj+({\tfrac{41995}{162409}}\eps+{\tfrac{479}{403}})\qk$ \\[0.5ex]
                            & $\displaystyle h_{1,3}  = 1+({\tfrac{12}{13}}+{\tfrac{72}{169}}\eps)\qi+({\tfrac{210}{169}}\eps-{\tfrac{4}{13}})\qj+({\tfrac{3}{13}}-{\tfrac{8}{169}}\eps)\qk$                           \\[1.0ex]
      $(M_1,M_3,M_2)\colon$ & $\displaystyle h_{2,1}  = 1+(\tfrac{5}{3}-\tfrac{5}{9}\eps)\qi+(\tfrac{1}{3}+{\tfrac{11}{9}}\eps)\qj+({\tfrac{14}{9}}\eps+\tfrac{1}{3})\qk$                                              \\[0.5ex]
                            & $\displaystyle h_{2,2}  = 1+({\tfrac{55}{39}}-{\tfrac{1324}{1521}}\eps)\qi+({\tfrac{38}{39}}+{\tfrac{814}{1521}}\eps)\qj+({\tfrac{748}{1521}}\eps+{\tfrac{56}{39}})\qk$                  \\[0.5ex]
                            & $\displaystyle h_{2,3}  = 1+({\tfrac{12}{13}}+{\tfrac{72}{169}}\eps)\qi+({\tfrac{210}{169}}\eps-{\tfrac{4}{13}})\qj+({\tfrac{3}{13}}-{\tfrac{8}{169}}\eps)\qk$                           \\[1.0ex]
      $(M_2,M_1,M_3)\colon$ & $\displaystyle h_{3,1}  = 1+({\tfrac{65}{31}}-{\tfrac{814}{961}}\eps)\qi+({\tfrac{16}{31}}+{\tfrac{1373}{961}}\eps)\qj+({\tfrac{1719}{961}}\eps+{\tfrac{18}{31}})\qk$                    \\[0.5ex]
                            & $\displaystyle h_{3,2}  = 1+({\tfrac{72}{217}}-{\tfrac{16813}{47089}}\eps)\qi+({\tfrac{136}{217}}-{\tfrac{7695}{47089}}\eps)\qj+({\tfrac{14752}{47089}}\eps+{\tfrac{153}{217}})\qk$      \\[0.5ex]
                            & $\displaystyle h_{3,3}  = 1+({\tfrac{11}{7}}+{\tfrac{10}{49}}\eps)\qi-({\tfrac{1}{7}}-{\tfrac{85}{49}}\eps)\qj+({\tfrac{5}{7}}-{\tfrac{5}{49}}\eps)\qk$                                  \\[1.0ex]
      $(M_2,M_3,M_1)\colon$ & $\displaystyle h_{4,1}  = 1+\qi+\eps\qj+\eps\qk$                                                                                                                                         \\[0.5ex]
                            & $\displaystyle h_{4,2}  = 1+({\tfrac{10}{7}}-{\tfrac{59}{49}}\eps)\qi+({\tfrac{8}{7}}+{\tfrac{13}{49}}\eps)\qj+({\tfrac{54}{49}}\eps+{\tfrac{9}{7}})\qk$                                 \\[0.5ex]
                            & $\displaystyle h_{4,3}  = 1+({\tfrac{11}{7}}+{\tfrac{10}{49}}\eps)\qi-({\tfrac{1}{7}}-{\tfrac{85}{49}}\eps)\qj+({\tfrac{5}{7}}-{\tfrac{5}{49}}\eps)\qk$                                  \\[1.0ex]
      $(M_3,M_1,M_2)\colon$ & $\displaystyle h_{5,1}  = 1+(\tfrac{5}{3}-\tfrac{5}{9}\eps)\qi+(\tfrac{1}{3}+{\tfrac{11}{9}}\eps)\qj+({\tfrac{14}{9}}\eps+\tfrac{1}{3})\qk$                                              \\[0.5ex]
                            & $\displaystyle h_{5,2}  = 1+(\tfrac{1}{3}-\tfrac{4}{9}\eps)\qi+(\tfrac{2}{3}-\tfrac{2}{9}\eps)\qj+(\tfrac{2}{3}+\tfrac{4}{9}\eps)\qk$                                                    \\[0.5ex]
                            & $\displaystyle h_{5,3}  = 1+2\qi+2\eps\qj+\qk$                                                                                                                                           \\[1.0ex]
      $(M_3,M_2,M_1)\colon$ & $\displaystyle h_{6,1}  = 1+\qi+\eps\qj+\eps\qk$                                                                                                                                         \\[0.5ex]
                            & $\displaystyle h_{6,2}  = 1+(1-\eps)\qi+\qj+(1+\eps)\qk$                                                                                                                                 \\[0.5ex]
                            & $\displaystyle h_{6,3}  = 1+2\qi+2\eps\qj+\qk$                                                                                                                                           \\
  \end{supertabular}
  \par\medskip
  The rotation quaternions $h_{l,i}$ correspond to edges of the
  linkage graph in Figure~\ref{fig:linkage_graph}. The six rotation
  quaternions $h_{l,2}$ are all different. The six rotation
  quaternions $h_{l,1}$ come in three pairs, corresponding to
  permutations of the shape $(M_i,M_j,M_k)$ and $(M_j,M_i,M_k)$.
  Similarly, equal rotation quaternions $h_{l,3}$ come from
  permutations $(M_i,M_j,M_k)$ and $(M_i,M_k,M_j)$, respectively.
\end{example}

\begin{rem}
  There exist cases where two consecutive rotation quaternions, like
  $h_{6,1}$ and $h_{6,2}$ in the above example, are compatible. In
  this case, the linkage graph of Figure~\ref{fig:linkage_graph} is
  still correct but with the understanding that the revolute axes to
  $h_{6,1}$ and $h_{6,2}$ coincide (compare
  Remark~\ref{rem:coinciding}). It is possible to reflect this in the
  linkage graph by removing the edge labeled $h_{6,2}$ and inserting
  an additional edge connecting $\{\}$ and~$\{1,2\}$. The relative
  motion between the the links $\{\}$ and $\{1,2\}$ is then the
  product of two linearly parametrized rotations about the same axis,
  that is, a quadratically parametrized rotation.
\end{rem}

From the description of the linkage graph given in
Theorem~\ref{thm:closed} we can draw even more conclusions:
\begin{itemize}
\item Every 4-cycle in the mechanism corresponds to a Bennett linkage
  (general case), to a planar or spherical 4-bar linkage with a
  rational coupler curve, or to four joints sharing a common revolute
  axis.
\item Any of the $n!$ paths of length $n$ that connects the base
  $\{\}$ with the platform $\{1,\ldots,n\}$ corresponds to an open
  $n$R-chain. Just as any two permutations can be transformed into
  each other by transpositions, any two of these $n$R-chains are
  connected by a series of Bennett substitutions, where two
  consecutive revolute axes $H_{l,i}$, $H_{l,i+1}$ are replaced by two
  axes $H_{m,i}$, $H_{m,i+1}$ that complete the first axis pair to the
  axis quadruple of a Bennett mechanism.
\item The mechanism's theoretical degree of freedom, computed
  according to the formula of Chebychev-Grübler-Kutzbach, is
  $f_n=6(2^n-1)-5n2^{n-1}$. Thus, $f_2 = -2$, $f_3 = -18$, $f_4 =
  -70$, $f_5 = -214$ etc. with rapid decrease as $n$ goes to $\infty$.
\item Our construction stays within the planar motion group
  $\mathrm{SE}(2)$ or the spherical motion group $\mathrm{SO}(3)$.
  That is, if the polynomial $P$ describes a planar or spherical
  motion, all rotation quaternions obtained from different
  factorizations of $P$ belong to the same group. Thus, we can
  construct planar or spherical linkages with the same combinatorics
  as in the spatial case but with a theoretical degree of freedom of
  $g_n = 3(2^n-1)-n2^n$, that is $g_2 = 1$, $g_3 = -3$, $g_4 = -19$,
  $g_5 = -67$ etc. One example is depicted in Figure~\ref{fig:planar}.
  It can be thought of as a linkage which is composed of two planar
  3RRR-platforms with identical anchor points on the base ($B_0$,
  $B_1$, $B_2$) and on the platform ($P_0$, $P_1$, $P_2$). In
  addition, there exists a hexagon formed by the middle revolute
  joints, whose side lengths remain constant during the motion. In
  Figure~\ref{fig:planar}, the hexagon sides are drawn as double
  lines. An animation of this linkage can be found on the accompanying
  web-site \url{http://geometrie.uibk.ac.at/schroecker/qf/}. Its
  theoretical degree of freedom is~$-3$.
\end{itemize}

\begin{figure}
  \centering
  \includegraphics{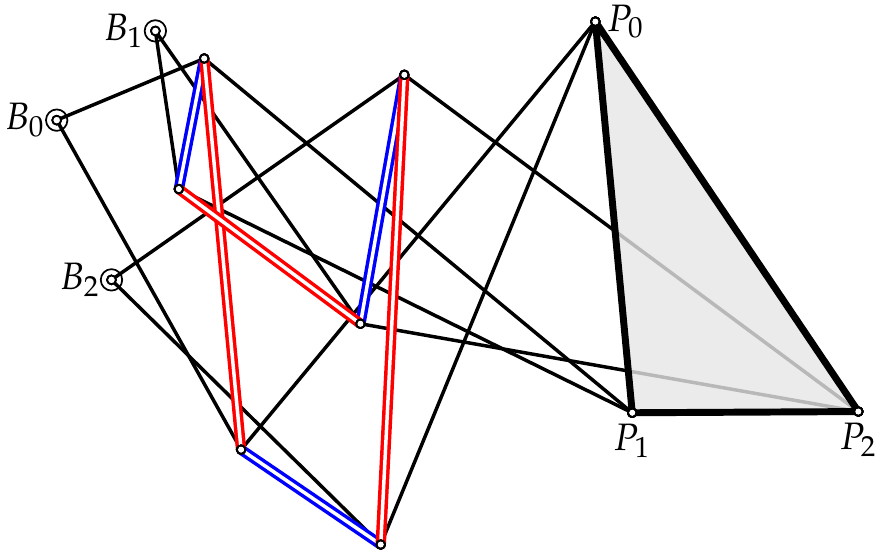}
  \caption{Planar overconstrained linkage}
  \label{fig:planar}
\end{figure}

\section{Translation quaternions and prismatic joints}
\label{sec:prismatic-joints}

Theorem~\ref{thm:closed} admits a small but interesting generalization
that allows for a mixture of translational and rotational joints. In
fact, the factorization Algorithm~\ref{alg:factor} also works under
slightly more general assumptions: We can allow that the norm
polynomial $P\qe{P}$ has a factorization into $n$ distinct monic
quadratic polynomials $M_1,\ldots,M_n$, each of which is either
irreducible or the square of a linear polynomial. In this case, the
assumptions of Lemma~\ref{lem:root} are fulfilled and we again get
$n!$ factorizations of $P$ into linear motion polynomials.

If $M$ is a square factor of the norm polynomial $P\qe{P}$, the factor
$(t-h)$ satisfying $M = (t-h)(t-\qe{h})$ parametrizes a translational
one-parameter subgroup. The primal part of $h$ is a scalar, say
$\lambda$, and $M=(t-\lambda)^2$; the dual part of $H$ is purely
vectorial and specifies the direction of the translation.

If the motion polynomial $P$ has a factorization $P\qe{P}\in \R[t]$ as
$P = (t-h_1) \cdots (t-h_n)$ such that the minimal polynomials
$M_1,\ldots,M_n$ of $h_1,\dots,h_n$ are all distinct, we can find this
factorization by Algorithm~\ref{alg:factor}, regardless whether
$h_1,\ldots,h_n$ represent rotations or translations. As in the purely
rotational case, there will be exactly $n!$ such factorizations.

If the factorization $P = (t-h_1)\cdots(t-h_n)$ contains a translation
polynomial $h_i$ with primal part $p_i$, the point $P(p_i)$ lies in
the exceptional three-space $E$. Conversely, if the curve $C$
intersects $E$ and admits a factorized parametrization $P =
(t-h_1)\cdots(t-h_n)$ with rotation or translation polynomials
$h_1,\ldots,h_n$, $\primal(P)$ must vanish for some parameter value
$p$. This implies that one of the factors $(p-h_1),\ldots,(p-h_n)$ has
zero primal part, that is, $\primal(h_i)=p \in \R$ and $h_i$ is a
translation polynomial.

We illustrate this at hand of two examples.

\begin{example}
  Assume that $P(t) = (t-h_1) \cdots (t-h_n)$ and all quaternions
  $h_1,\ldots,h_n$ are translation quaternions. This means, that the
  motion is a pure translation which can be described by the sum of
  vectors $s(t) = (u_1 + tv_1) + \cdots + (u_n + tv_n)$ with $u_i,v_i
  \in \R^3$. But $s(t)$ also equals the sum of any permutation of the
  summands $u_i + tv_i$. This trivial statement is a special case of
  Theorem~\ref{thm:number}.
\end{example}

\begin{example}
  \label{ex:rprp}
  The polynomial
  \begin{equation*}
    P = t^2 - t(2 + (1-\eps)\qi + (1+\eps)\qj +
    (1+2\eps)\qk) + 1-2\eps + (1-\eps)\qi + (1+2\eps)\qj + (1+\eps)\qk
  \end{equation*}
  admits the factorizations $P = (t-h_1)(t-h_2) = (t-k_1)(t-k_2)$
  where
  \begin{equation*}
    \begin{gathered}
      h_1 = 1 + \frac{1}{3}((3-7\eps)\qi + (3+2\eps)\qj + (3+5\eps)\qk),\quad
      h_2 = 1 + \frac{\eps}{3}(4\qi + \qj + \qk)\\
      k_1 = 1 + \eps\qj + \eps\qk,\quad
      k_2 = 1 + (1-\eps)\qi + \qj + (1+\eps)\qk.
    \end{gathered}
  \end{equation*}
  Both, $k_1$ and $h_2$ are translation quaternions. Indeed, $P(1) =
  \eps(-2+\qj+\qk) \in E$, that is, the curve $C$ with parametrization
  $P$ intersects the exceptional three-space~$E$.
\end{example}

The kinematic interpretation of Example~\ref{ex:rprp} is well-known.
The motion $C$ is the coupler motion of an RPRP linkage \cite{Perez}
and can be obtained as the limiting case of the Bennett motion The
observations made in Example~\ref{ex:rprp} immediately generalize to
higher degree polynomials.

\section{New overconstrained 6R-chains}
\label{sec:overconstrained-6r}

For $n = 3$, the mechanism described in the preceding section contains
a number of overconstrained 6R-chains, some of which are new. This is
also true for certain examples in case of $n=4$. Although only a side
result, the discovery of new overconstrained 6R-chains is an important
contribution of this article. Therefore we present a more detailed
discussion.

\subsection{Coupler curves of degree three}

Here, we assume $n=3$. The link graph is the 1-skeleton of a
three-dimensional cube. The relative motion between the link labeled
by the empty set $F=\{\}$ (the base) and the link labeled
$G=\{1,2,3\}$ (the platform) is a rational cubic curve $C \subset S$
in the Study quadric. In general, there is nothing special about links
$F$ and $G$. We could as well take any two links which are diagonally
opposite in the cube.

Every path of length three that connects $F$ and $G$ corresponds to an
open 3R-chain which is capable of performing the motion $C$. Combining
two of these paths yields an overconstrained 6R-chain with coupler
motion $C$. There is a total of six paths of lengths three connecting
$F$ and $G$ and a total of 15 possible combinations of two such paths.
They can be classified into three types of different combinatorics
(Figure~\ref{fig:cube_paths}) or, equivalently, into different types
of pairs of permutations. We say that two permutations $\sigma_1$ and
$\sigma_2$ differ by $\sigma_2 \circ \sigma_1^{-1}$.
\begin{itemize}
\item The first type (a) corresponds to permutations which differ by a
  neighbor transposition. The resulting 6R-chain is trivial. It
  consists of a Bennett linkage plus one dangling link. The complete
  mechanism contains six sub-linkages of this type.
\item The second type (b) corresponds to permutations that differ by a
  cyclic permutation. The resulting 6R-chain is Waldron's double
  Bennett hybrid \cite[pp.~63--65]{dietmaier95}. There are six
  sub-linkages of this type.
\item The third type (c) is new. It corresponds to permutations that
  differ by the permutation $(3,2,1)$. Three of the 15 6R-chains are
  of this type.
\end{itemize}

\begin{figure}
  \centering
  \includegraphics{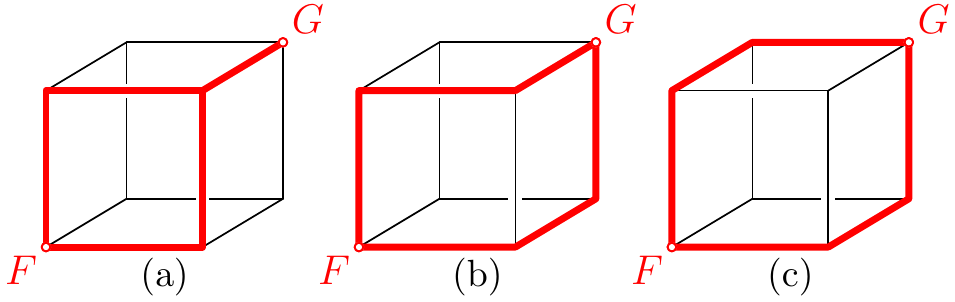}
  \caption{Different types of closed 6R-chains: Bennett linkage plus
    dangling link (trivial) (a), Waldron's double Bennett hybrid (b),
    and new (c)}
  \label{fig:cube_paths}
\end{figure}

Currently, overconstrained 6R-chain are classified by relations
between their Denavit-Hartenberg parameters (distance, angle, and
offset) on the cyclic sequence of revolute axes, see for example the
list at the end of \cite{B09}. The reasons for their mobility are of
geometric nature. Often, they are inferred from the mobility of
related structures that are already known to be flexible. Our newly
found examples are mobile for \emph{algebraic reasons} or, more
precisely, our construction and proof of mobility is algebraic.
Therefore, we do not offer simple relations between their
Denavit-Hartenberg parameters. However, it is possible to compute
exact symbolic expressions of these parameters for numeric examples.
Below we list the angles, distances and offsets for the three new
overconstrained 6R-chains obtained from Example~\ref{ex:3}.

\begin{description}
\item[] Linkage
  $h_{1,3}$, $h_{1,2}$, $h_{1,1}$, $\qe{h}_{6,1}$, $\qe{h}_{6,2}$, $\qe{h}_{6,3}$
  \par
  Distances: $\frac{16\sqrt{29}}{377}$, $\frac{\sqrt{1115179082}}{63302}$, $\frac{37\sqrt{854}}{1586}$, $\frac{24\sqrt{145}}{899}$, $\frac{\sqrt{2}}{2}$, $\frac{\sqrt{6}}{6}$
  \par
  Offsets: $\frac{7945}{59218}$, $\frac{38174\sqrt{3}}{62281}$, $\frac{545\sqrt{5}}{3538}$, $\frac{7}{58}$, $\frac{2\sqrt{3}}{3}$, $\frac{11\sqrt{5}}{58}$
  \par
  Angle cosines: $\frac{27\sqrt{5}}{65}$, $\frac{29\sqrt{3}}{93}$, $\frac{41\sqrt{15}}{195}$, $\frac{13\sqrt{5}}{31}$, $\frac{\sqrt{3}}{3}$, $\frac{\sqrt{15}}{5}$
\item[]Linkage
  $h_{2,3}$, $h_{2,2}$, $h_{2,1}$, $\qe{h}_{4,1}$, $\qe{h}_{4,2}$, $\qe{h}_{4,3}$
  \par
  Distances: $\frac{\sqrt{1115179082}}{185822}$, $\frac{8\sqrt{29}}{87}$, $\frac{37\sqrt{854}}{1586}$, $\frac{\sqrt{2}}{6}$, $\frac{12\sqrt{145}}{203}$, $\frac{\sqrt{6}}{6}$
  \par
  Offsets: $\frac{7945}{59218}$, $\frac{1765\sqrt{5}}{3538}$, $\frac{16\sqrt{3}}{61}$, $\frac{7}{58}$, $\frac{31\sqrt{5}}{58}$, $\frac{968\sqrt{3}}{3063}$
  \par
  Angle cosines: $\frac{151\sqrt{3}}{273}$, $\frac{4\sqrt{5}}{15}$, $\frac{41\sqrt{15}}{195}$, $\frac{5\sqrt{3}}{9}$, $\frac{2\sqrt{5}}{7}$, $\frac{\sqrt{15}}{5}$
\item[]
  Linkage $h_{3,3}$, $h_{3,2}$, $h_{3,1}$, $\qe{h}_{5,1}$, $\qe{h}_{5,2}$, $\qe{h}_{5,3}$
  \par
  Distances: $\frac{\sqrt{6}}{42}$, $\frac{\sqrt{1115179082}}{63302}$, $\frac{12\sqrt{145}}{203}$, $\frac{37\sqrt{854}}{11346}$, $\frac{\sqrt{2}}{2}$, $\frac{8\sqrt{29}}{87}$
  \par
  Offsets: $\frac{968\sqrt{3}}{3063}$, $\frac{53315}{59218}$, $\frac{545\sqrt{5}}{3538}$, $\frac{16\sqrt{3}}{61}$, $\frac{53}{58}$, $\frac{11\sqrt{5}}{58}$
  \par
  Angle cosines: $\frac{9\sqrt{15}}{35}$, $\frac{29\sqrt{3}}{93}$, $\frac{2\sqrt{5}}{7}$, $\frac{359\sqrt{15}}{1395}$, $\frac{\sqrt{3}}{3}$, $\frac{4\sqrt{5}}{15}$
\end{description}

\subsection{Coupler curves of degree four}

Now we consider a second construction of new overconstrained 6R-chains
whose coupler curve is of degree four ($n=4$). Let $h_1,h_2,h_3,h_4$
be dual quaternions representing rotations. Set
$P=(t-h_1)(t-h_2)(t-h_3)(h-h_4)$ and $M_i:=(t-h_i)(t-\qe{h_i})$ for
$i=1,2,3,4$. The link graph of the linkage constructed in
Theorem~\ref{thm:closed} has 16 links and 32 joints. Any open chain
connecting base and platform has four links, any closed chain composed
of two open chains of this type has eight links. In order to get a
6R-chain, we assume that $(h_1,h_2)$ and $(h_3,h_4)$ are compatible
pairs. Now we use Algorithm~\ref{alg:factor} to compute another
factorization $P=(t-g_1)(t-g_2)(t-g_3)(t-g_4)$. This gives rise to an
8R-linkage. But we can remove the link between $h_1$ and $h_2$ and the
link between $h_3$ and $h_4$ in order to get a 6R-linkage. In contrast
to the previous examples, the relative motion of some neighbouring
links is rational of degree two.

The second factorization depends on the chosen permutation of the
factors $M_i$, $i=1,\dots,4$. Of course, using the permutation
$(M_1,M_2,M_3,M_4)$ is prohibited. Because of compatibility relations,
we also have
\begin{equation*}
  P=(t-h_2)(t-h_1)(t-h_3)(t-h_4)=(t-h_1)(t-h_2)(t-h_4)(t-h_3).
\end{equation*}
Hence, we must avoid common right or left factors with these two
factorizations as they lead to dangling links. In total, there are
$16$ permutations of this type, because they start with $M_1$ or
$M_2$, or end with $M_3$ or $M_4$, or both.

A possible choice of a permutation is $(M_3,M_1,M_4,M_2)$. The
resulting linkage is known as ``serial Goldberg 6R linkage'' (see
\cite[4.3.1]{dietmaier95}). There are four permutations that lead to
this type of linkage.

Another possible choice is $(M_3,M_4,M_1,M_2)$. The resulting linkage
is of a new type and, again, can be obtained from four different
permutations. There are some apparent relations between the
Denavit-Hartenberg parameters, which can be explained by the fact that
the axes of three joints of the 6R-linkage are axes of a 4-cycle of
the full ``hypercube'' linkage. But these relations are not sufficient
for the mobility.

\section{Conclusion}
\label{sec:conclusion}

Factorization of motion polynomials is a new and powerful method for
synthesizing linkages with a prescribed rational coupler motion. We
gave an explicit construction, based on dual quaternion factorization,
for such linkages. It allows to incorporate translation quaternions
(prismatic joints) and compatible quaternions (high degree relative
rotations of neighboring links).

We studied the resulting linkages and explained how to obtain new
types of closed 6R-chains from them. We provided ideas for their
construction, but the method leaves room for creativity which may lead
to more new types. For example, it is possible to construct
overconstrained chains with $2k$ prismatic and $6-2k$ revolute joints
by combining the ideas of Sections~\ref{sec:prismatic-joints} and
\ref{sec:overconstrained-6r}. Coinciding consecutive joints will
produce overconstrained chains consisting of five joints (Goldberg
linkages).

An important new insight is that flexibility of overconstrained
6R-chains has rather algebraic than geometric reasons. Currently, the
authors work on a systematic description and computational tools for
synthesizing linkages based on this new method.

\section*{Acknowledgements}

This research was supported by the Austrian Science Fund (FWF):
I~408-N13 and DK~W~1214-N15.

\bibliographystyle{elsarticle-num-names}
\bibliography{ark2}

\end{document}